\documentclass[a4paper,11pt]{article}

\usepackage[utf8]{inputenc}
\usepackage[english]{babel}
\usepackage{fullpage}
\usepackage{xspace}
\usepackage{amsmath}
\usepackage{amsthm}
\usepackage{amssymb}
\usepackage[bookmarks=true,
            bookmarksnumbered=true,
            pagebackref=true,
            colorlinks=true,
            linkcolor=blue,
            citecolor=red]{hyperref}
\title{A note on Hindman-type theorems for uncountable cardinals}

\author{Lorenzo Carlucci \\
  Department of Computer Science, University of Rome I.\\
  \texttt{carlucci@di.uniroma1.it}}

\date{\today{}}

\hypersetup{
  pdfauthor={Lorenzo Carlucci}
  pdftitle ={}
}

\newtheorem{proposition}{Proposition}
\newtheorem{theorem}{Theorem}
\newtheorem{corollary}{Corollary}

\newtheorem{open.problem}{Open Problem}

\newtheorem*{theorem*}{Theorem}
\newtheorem*{corollary*}{Corollary}
\newtheorem*{proposition*}{Proposition*}
\newtheorem*{lemma*}{Lemma}
\newtheorem*{fact*}{Fact}
\newtheorem*{claim*}{Claim}
\newtheorem*{open.problem*}{Open Problem}
\newtheorem*{remark*}{Remark}
\newtheorem*{example*}{Example}
\newtheorem*{exercise*}{Exercise}
\begin{document} % ----------- BEGIN DOCUMENT -------------------------

% ---------------------------- TITLE/ABSTRACT PAGE --------------------

\maketitle

\begin{abstract}
Recent results of Hindman, Leader and Strauss and of Fern\'andez-Bret\'on and Rinot showed that natural versions of Hindman's Theorem fail {\em for all} uncontable cardinals. 
On the other hand, Komj\'ath proved a result in the positive direction, 
showing that {\em there are} arbitrarily large abelian groups satisfying {\em some} Hindman-type property. 
In this note we show how a family of natural Hindman-type theorems for uncountable cardinals can be obtained by adapting
some recent results of the author from their original countable setting. We also show how lower bounds for {\em some} of 
the variants considered can be obtained. 
\end{abstract}

% Uncomment this if you want the paper to start in the next page.
%\thispagestyle{empty}
%\clearpage
% ---------------------------------------------------------------------

% --------------------------- CONTENT --------------------------------

\section{Introduction}

Hindman's Finite Sums (or Finite Unions) Theorem~\cite{Hin:74} is a fundamental result in Ramsey Theory. 
It can be stated as follows (see~\cite{Baum:74}): If the finite subsets of $\omega$ are colored in finitely many 
colors, then there is an infinite family of pairwise disjoint finite subsets of $\omega$ such that
all elements and all finite unions of elements of the family have the same color. 

The question of whether analogues of Hindman's Theorem hold for uncountable cardinals goes back to Erd\H{o}s~\cite{Erd:73}. Erd\H{o}s' questions were about very strong analogues of Hindman's Theorem
(finite partitions of all subsets of a cardinal) and were 
answered in the negative, while some weakened versions were proved to hold 
(see~\cite{Ele-Haj-Kom:91} for a discussion and bibliography). 

This theme has been interestingly revived in very recent times by David Fern\'andez-Bret\'on~\cite{Bre:16}, 
who showed that some very natural analogues of Hindman's Theorem (weaker than Erd\H{o}s' versions) 
fail {\em for all\/} uncountable cardinals. More precisely, let 
$\mathrm{HIND}(\kappa,\lambda)$ indicate that every coloring in $2$ colors of the finite subsets of
a $\kappa$-sized set admits 
an $\lambda$-sized family of finite subsets such that all its members and all their finite unions have the same
color. 
The main result of~\cite{Bre:16} is that $\mathrm{HIND}(\kappa,\lambda)$ fails whenever 
$\aleph_0 < \lambda \leq \kappa$. 
Further negative results were proved by Fernand\'ez-Bret\'on and Rinot in a follow-up paper~\cite{Bre-Rin:16}, where, e.g., the 
following theorem is established: Every uncountable Abelian
group can be colored with countably many colors, so that no color class contains
uncountably many elements with all their subsums.
By contrast, Komj\'ath~\cite{Kom:16}  very recently showed results in the positive direction, proving that {\em there exist} arbitrarily large cardinals satisfying {\em some form} of Hindman-type theorem.
In particular: For every finite $n$ and infinite cardinal $\kappa$ there is an Abelian group $G$
such that for every coloring of $G$ in $\kappa$ colors there are $n$ elements so that
the same color class contains all subsums of them. Moreover: For every cardinal $\kappa$, for every $n>1$,
there exists an Abelian group $G$ such that for every coloring of $G$
in $\kappa$ colors there are distinct elements $\{a_{i,\alpha} : 1 \leq  i \leq  n, \alpha < \kappa\}$ 
such that all sums of the form $a_{i_1,\alpha_1} +\dots + a_{i_r,\alpha_r}$ ($i_1, i_2,\dots, i_r$ different) 
are distinct and in the same color class. Note that Komj\'ath's results are existential.
 
The purpose of this note is to show that some results from the author's~\cite{Car:16} about variants of 
the (countable) Finite Sums Theorem can be transferred to uncountable 
cardinals -- by essentially the same arguments -- so as to obtain a family of positive examples of uncountable
Hindman-type theorems. The role played by the countable Ramsey's Theorem in~\cite{Car:16} 
is here played by Erd\H{o}s-Rado Theorem. We further show that, for {\em some} of our Hindman-type theorems, a lower bound can be established by reduction to a Ramsey property.
%Our results deal with Hindman-type conditions that are weaker than Komj\'ath's but 
%that hold for all sufficiently large semigroups, rather than for some particular ones. 
%For some particular semigroup we also obtain that some of the Hindman-type conditions considered
%admit a lower bound in terms of Ramsey-theoretic properties.
A typical member of the family is the following {\em uncountable Hindman-Schur Theorem}: 
For every $\lambda$, for every positive integer $c$, there exists $\kappa$ such that
{\em for any} commutative semigroup $(G,+)$ of size $\kappa$ for every coloring of $G$ in $c$ colors there is 
a $\lambda$-sized $X\subseteq G$ and positive integers $a,b$ such that all $a$-term, $b$-term and $(a+b)$-term 
sums of elements of $X$ have the same color. 
Note that the theorem is universal. 
%Furthermore, in case $(G,+)$ is $([\kappa]^{<\omega},\cup)$, we show that the solution family $X\subseteq[\kappa]^{<\omega}$ can be ensured 
%to be a block sequence (in the sense of~\cite{Arg-Tod:05}, see below for a definition).
%We can show that in this case one obtains a lower bound, i.e., 
%such a $\kappa=\kappa(\lambda)$ necessarily satisfies $\kappa \to (\lambda)^2_2$.

\section{Upper bounds} 

We establish a number of reasonably natural variants of Hindman's Theorem for uncountable cardinals. 
These variants relax the monochromaticity condition to those sums whose length belongs to {\em some
structured set} of positive integers.  
The proofs are a literal translation to the uncountable setting of proofs from the author's~\cite{Car:16}. 
Where the countable Ramsey's Theorem is used in~\cite{Car:16}, we here use the Erd\H{o}s-Rado Theorem.

Let us fix some terminology and notation and recall some basic facts. 
If $\kappa$ is a cardinal we denote by $\kappa^+$ its successor. We denote by $[\kappa]^{<\omega}$
the set of finite subsets of $\kappa$.
Let $\beth_0(\kappa)=\kappa$ and $\beth_{n+1}(\kappa)=2^{\beth_n(\kappa)}$.
Erd\H{o}s-Rado Theorem (see~\cite{EHMR:84}) is the following statement: For every infinite cardinal $\kappa$, for every
integer $n$, the following holds: If the $(n+1)$-tuples of a set $X$ of cardinality $\beth_n(\kappa)^+$ are 
colored in $\kappa$ colors then there exists a subset $H$ of $X$ such that $H$ has cardinality 
$\kappa^+$ and such that all $(n+1)$-tuples from $H$ have the same color. This fact is written as 
$\beth_n(\kappa)^+ \to (\kappa^+)^{n+1}_\kappa$ in partition calculus notation. 
In this notation $\kappa \to (\lambda)^2_2$ indicates the following Ramsey property: 
If the pairs of a set of cardinality $\kappa$ are
colored in two colors then there exists a monochromatic subset of size $\lambda$.
We will only use the Erd\H{os}-Rado Theorem for
colorings with finitely many colors. If $A$ is a set of positive integers and $X$ is any set, and $f$ is a binary 
operation defined on $X$, we denote by 
$FS^A(X)$ the set of all $f$-sums of distinct $j$-many elements from $X$, for $j\in A$. In the particular case
that $f$ is the union, we write $FU^A(X)$. 

We start with an example, called the {\em uncountable Hindman-Van Der Waerden Theorem}. 

\begin{theorem}\label{thm:hindmanVDW}
For every infinite cardinal $\lambda$, for every positive integers $c$ and $d$, 
there exists a cardinal $\kappa>\lambda$ such that for any commutative
semigroup $(G,+)$ of size $\kappa$ the following holds: 
For every coloring of $G$ in $c$ colors there exists a subset $H\subseteq G$ of cardinality $\lambda$
and positive integers $a,b$ such that $FS^{\{a,a+b,\dots,a+d\cdot b\}}(H)$ is monochromatic.
\end{theorem}

\begin{proof}
Let $\lambda$, $d$ and $c$ be given. By Van der Waerden's Theorem~\cite{VdW:27} let $n$ be so large that every coloring 
of $[1,n]$ in $c$ colors admits a monochromatic 
arithmetic progression of length $d$. Let $\kappa$ witness the Erd\H{o}s-Rado Theorem
for colorings of $n$-tuples in $c^n$ colors and target size $\lambda$ for the homogeneous set. 
Let $(G,+)$ be a commutative semigroup of 
cardinality $\kappa$. Let $f:G\to c$ be a coloring. Let $\{x_\alpha \,:\, \alpha < \kappa\}$
be an injective enumeration of $G$. Define $F:[\kappa]^n\to c^n$ as follows:
$$ F(\alpha_1,\dots,\alpha_n) := \langle f(x_{\alpha_1}),f(x_{\alpha_1}+x_{\alpha_2}),\dots,f(x_{\alpha_1}+\dots + x_{\alpha_n})\rangle.$$
By Erd\H{o}s-Rado Theorem let $X$ be a subset of $\kappa$ of cardinality $\lambda$ such that
$[X]^n$ is monochromatic of color $i<c^n$ for $F$. Let $H=\{ x_\alpha \,:\, \alpha \in X\}$.
For each $1\leq m\leq n$, all sums of $m$ many terms from $H$ have the same color $i_m < c$ under $f$.
Thus we have an induced coloring of $[1,n]$ in $c$ colors. By choice of $n$ there exist positive 
integers $a,b$ such that $a,a+b,\dots,a+d\cdot b$ is a monochromatic arithmetic progression. 
\end{proof}

As is the case in the countable setting (see\cite{Car:16}), a host of variations can be proved by the same 
argument, by applying a different finite combinatorial theorem in place of Van der Waerden's Theorem.
For example, using the Folkman-Rado-Sanders' Theorem~\cite{San:68} 
we obtain the following {\em uncountable Hindman-Folkman's Theorem}.

\begin{theorem}\label{thm:hindmanF}
For every infinite cardinal $\lambda$, for every positive integers $c$ and $d$, 
there exists a $\kappa>\lambda$ such that for any commutative
 semigroup $G$ of cardinality $\kappa$ the following holds: 
If $f$ colors $G$ in $c$ colors there exists a subset $H\subseteq G$ of cardinality $\lambda$
and positive integers $a_1<\dots < a_d$ such that $FS^{FS(\{a_1,\dots,a_d\})}(H)$ is monochromatic.
\end{theorem}

%Theorem \ref{thm:hindmanF} can be compared with Theorem 1 of~\cite{Kom:16}. The latter is stronger
%in that it works for $\kappa$-colorings. While Theorem 1 of~\cite{Kom:16} is existential (i.e., proves the 
%existence of an abelian group satisfying some Hindman-type homogeneity condition), Theorem 
%\ref{thm:hindmanF} is universal (i.e., the Hindman-type condition is satisfied by any large enough 
%semigroup).

By contrast with~\cite{Car:16}, in the present setting we have not to worry about the provability 
of the relevant theorems from finite combinatorics (Van Der Waerden's Theorem, Schur's Theorem, Folkman's Theorem) 
in a weak theory of arithmetic. For example, using a theorem by Bigorajska and Kotlarski~\cite{Big-Kot:99} on partitioning $\alpha$-large sets for $\alpha<\varepsilon_0$, we obtain the following Hindman-type theorem.

\begin{theorem}
For every infinite cardinal $\lambda$, for every ordinal $\beta < \epsilon_0$, for every positive integers $m$ and $d$,
there exists a $\kappa>\lambda$ such that for any commutative
semigroup $G$ of cardinality $\kappa$ the following holds: 
If $f$ colors $G$ in $d$ colors there exists a subset $H\subseteq G$ of cardinality $\lambda$
and an $\omega^\beta$-large set $B$ 
of positive integers, with $\min(B)\geq m$, such that $FS^{B}(H)$ is monochromatic.
\end{theorem}

\begin{proof}
From Theorem 1 of~\cite{Big-Kot:99} it follows that if an $\omega^\beta\cdot (d+1)$-large set is partitioned
in $d$ pieces then one of the pieces is $\omega^\beta$-large. We start by picking $n$ so large that
$[m,n]$ contains an $\omega^\beta\cdot (d+1)$ subset, say $X$. We then argue as in
the proof of Theorem \ref{thm:hindmanVDW}, applying the Bigorajska-Kotlarski's Theorem to obtain 
$\omega^\beta$-large homogeneous subset of $X$
for the coloring induced on $X\subseteq [m,n]$ by the coloring of $[1,n]$ defined as in the proof of
Theorem \ref{thm:hindmanVDW}.
\end{proof}

In the particular case of the semigroup consisting of $[\kappa]^{<\omega}$ with union, we can 
additionally ensure an ``unmeshedness'' or ``block'' condition on the solution set. Two finite non-empty 
subsets $x,y$ of some set $X$ are {\em unmeshed} if either $\max{x}<\min{y}$ or 
$\max{y}<\min{x}$. A family of finite subsets of some set $X$ is called {\em pairwise unmeshed} 
if its members are pairwise unmeshed. This notion is equivalent to the notion of a {\em block sequence}
(see, e.g.,~\cite{Arg-Tod:05}, Chapter III).

If $T$ is some Hindman-type theorem, we indicate by ``block $T$'' or by ``$T$ with unmeshedness''
the strengthening of $T$ obtained by imposing that the solution set is a block sequence. This literally 
makes sense only when dealing with colorings of $[X]^{<\omega}$ for some $X$, although it could
be generalized to other settings. 
By very slightly adapting the proof of Theorem \ref{thm:hindmanVDW}, we obtain the following 
{\em uncountable Hindman-Schur theorem with unmeshedness}.

\begin{theorem}\label{thm:unmeshed}
For every infinite cardinal $\lambda$, for every positive integer $c$, 
there exists a cardinal $\kappa>\lambda$ such that the semigroup $([\kappa]^{<\omega},\cup)$ satisfies
the following: 
If $f$ colors $[\kappa]^{<\omega}$ in $c$ colors then there exists a $\lambda$-sized family 
$H$ of finite subsets of $\kappa$ and positive integers $a,b$ such that $FU^{\{a,b,a+b\}}(H)$ is monochromatic,
and, furthermore, $H$ is a block sequence.
\end{theorem}

\begin{proof}
It is sufficient in the proof of Theorem \ref{thm:hindmanVDW} to apply the Erd\H{o}s-Rado theorem to a $\kappa$-sized
block sequence in $[\kappa]^{<\omega}$. 
\end{proof}

Note that this theorem is existential, i.e., the Hindman-type property holds for a particular semigroup of cardinality
$\kappa$, as is the case in Komj\'ath's results in~\cite{Kom:16}. 
We will show in the next section that unmeshedness of the solution set 
is useful to obtain lower bound on the Hindman-type theorems considered.
%For example, any $\kappa$ as in Theorem \ref{thm:unmeshed} satisfies $\kappa \to (\lambda)^2_2$. 

\section{Lower bounds}

We show a  lower bound (on $\kappa$ as a function of $\lambda$) on {\em some} of the above-considered
variants of Hindman's Theorem. The following proposition singles-out a weak condition for the lower bound to be
valid. 

\begin{proposition}\label{thm:lowbd}
Let $\lambda$ be an infinite cardinal. Let $\kappa$ be so large that it satisfies the following condition: For every coloring $c$ of $[\kappa]^{<\omega}$ in $3$ colors, there exists a $\lambda$-sized unmeshed family $H$ of finite subsets of $\kappa$ and an even integer $b>0$ such that $FU^{\{b\}}(H)$ is monochromatic under $c$. Then $\kappa \to (\lambda)^2_2$.
\end{proposition}

\begin{proof}
Let $d:[\kappa]^2\to 2$ be given. Define $c:[\kappa]^{<\omega}\to 3$ as 
follows. If $x$ is of odd cardinality, color $2$. Else $c(x)=d(\max(y),\max(z))$ where
$x = y \cup z$ such that $|y|=|z|$ and $y$ and $z$ are apart. Let $H$ be a $\lambda$-sized
unmeshed family of finite subsets of $\kappa$ and let $b$ be a positive even integer such that
$FU^{\{b\}}(H)$ is $c$-monochromatic. By the Pigeonhole Principle let $H'\subseteq H$ be a $\lambda$-sized family 
such that each element of $H'$ has the same cardinality, say $n$. Note that unmeshedness is preserved under
taking subsets.
Suppose without loss of generality that it is $b$. Suppose that $n$ is odd. Then it must be
the case that the color $2$ is never assigned to elements of $FU^{\{b\}}(H')$ since any 
$b$-sized union of elements of $H'$ has even cardinality. Suppose now that $n$ is even. Then 
again the color $2$ is never assigned to elements of  $FU^{\{b\}}(H')$. 
Let $i<2$ be the color of 
$FU^{\{b\}}(H')$ under $c$.
Now pick $X$ of size $\lambda$ as follows: $X$ consists of the max of elements of $FU^{\{b/2\}}(H')$ picked
in such a way as to satisfy that if $s,t\in X$, then $s=\max(y)$ for some $y\in FU^{\{b/2\}}(H')$ and $t=\max(z)$
for some $z \in FU^{\{b/2\}}(H')$ 
and $\max(y)<\min(z)$ or $\max(z)<\min(y)$.
We claim that $[X]^2$ is monochromatic for $d$ of color $i$. Let $s,t\in X$. Then for some $y,z\in FU^{\{b/2\}}(H')$, 
$\max(y)=s$ and $\max(z)=t$. Then $d(s,t)=d(\max(y),\max(z))=c(y\cup z)$ and $y \cup z\in FU^{\{b\}}(H')$.
Hence $d(s,t)=i$.
\end{proof}

The proof of Proposition \ref{thm:lowbd} is an adaptation of an argument attributed to Justin Tatch Moore which 
the author learned from David Fern\'andez-Bret\'on (personal communication), aimed 
at showing that any regular ``Hindman cardinal'' is weakly compact (a cardinal $\kappa$ is a {\em Hindman cardinal} if any finite coloring of $[\kappa]^{<\omega}$ admits a $\kappa$-sized family 
of finite subsets of $\kappa$ such that all members of the family and all finite unions thereof have the same color).

We say that a cardinal $\kappa$ {\em satisfies the Hindman-Schur condition with unmeshedness for a cardinal\/} $\lambda$
if the following holds: For every coloring $c$ of $[\kappa]^{<\omega}$
in $3$ colors, there exists a $\lambda$-sized unmeshed family $H$ of finite subsets of $\kappa$ and positive
integers $a,b$ such that $FU^{\{a,b,a+b\}}(H)$ is monochromatic under $c$. 
Theorem \ref{thm:unmeshed} above implies that for every infinite cardinal $\lambda$ there exists a $\kappa$ such that 
$([\kappa]^{<\omega},\cup)$ satisfies the Hindman-Schur condition with unmeshedness for $\lambda$.

\begin{corollary}
If $\kappa$ satisfies the Hindman-Schur condition with unmeshedness for $\lambda$, then ${\kappa \to (\lambda)^2_2}$.
\end{corollary}

\begin{proof}
Let $a,b$ be the positive integers witnessing that $\kappa$ satisfies the Hindman-Schur condition for 
$\lambda$. Just observe that one of $a,b,a+b$ is even, and invoke Proposition \ref{thm:lowbd}.
\end{proof}

Lower bounds on $\kappa= \kappa(\lambda)$ satisfying the Hindman-Schur Theorem with unmeshedness can then be read-off from known lower bounds on $\kappa=\kappa(\lambda)$ satisfying $\kappa \to (\lambda)^2_2$. For example, for any $\lambda$,
$\kappa=\kappa(\lambda^+)$ has to be larger than $2^\lambda$ (see~\cite{EHMR:84}). It is a natural question whether one can get similar lower bounds if the unmeshedness condition is dropped. 

\medskip
{\bf Acknowledgment} I thank David Fern\'andez-Bret\'on for his reading of and useful comments on a preliminary version of this note. 

\iffalse
\section{Skew Ramsey Theorem}

Let $\kappa \to^s (\lambda)^2_2$ indicate that for every 2-coloring $f$ of $[\kappa]^2$ there exist
a pair $(X_1,X_2)$ of subsets of $\kappa$ of size $\lambda$ such that every increasing pair
$\{\alpha_1,\alpha_2\}$ with $\alpha_1<\alpha_2$, $\alpha_i \in X_i$ has the same $f$-color. 
\fi

\end{document}